\documentclass[a4]{article} 

\usepackage{placeins} 
\usepackage[utf8]{inputenc} 
\usepackage{setspace} 
\usepackage{lineno} 
\usepackage{amsmath} 
\usepackage{amssymb} 
\usepackage{amsthm} 
\usepackage{pifont} 
\usepackage{bm} 
\usepackage[square,numbers]{natbib} 
\usepackage{titlesec} 
\usepackage[final]{pdfpages}
\titleformat*{\subsubsection}{\normalfont} 

\newtheorem{theorem}{Theorem}
\newtheorem{proposition}[theorem]{Proposition}
\newtheorem{lemma}[theorem]{Lemma} 
\graphicspath{{Figures/}} 

\setcitestyle{notesep={; },square,aysep={},yysep={;}}  

\title{Fibonacci groups, $F(2,n)$, are hyperbolic for $n$ odd and $n \geq 11$}
\author{Christopher P. Chalk}
\date{}

\begin{document}
\maketitle


\begin{abstract}

We prove that the Fibonacci group, $F(2, n)$, for $n$ odd and $n \geq 11$ is hyperbolic. We do this by  applying a curvature argument to an arbitrary van Kampen diagram of $F(2,n)$ and show that it satisfies a linear isoperimetric inequality. It then follows that $F(2, n)$ is hyperbolic.
\end{abstract}


\section{Introduction} 
The Fibonacci groups, $F(2, n)$, are $n$ generator $n$ relation groups defined symmetrically as 
\begin{align}
<x_1, x_2, . . . , x_n | x_i x_{i+1} = x_{i+2} \  i = 1,2, . . . ,n>
\end{align}
where the suffixes are taken mod $n$.
For $n = 2m$ and $m \geq 4$, these groups have been shown to be fundamental groups of certain hyperbolic manifolds. See \cite{hkm}. The automatic groups software package, KBMAG, \cite{kbm}, can show that $F(2,9)$ is hyperbolic. See \cite{holt}, section 13.4.  The group $F(2,6)$ is known to contain a free abelian subgroup of rank 2 and so is not  hyperbolic. All other groups, when $n < 8$, are finite. See \cite{dlj} section 16.4.  We prove, here, that $F(2,n)$ for $n$ odd and $n \geq 11$ is hyperbolic. 
The proof is in two parts. The first part, in Lemmas 1 to 4, classifies the vertices of any van Kampen diagram, K, of $F(2,n)$ into various types based on the coloured edge rules defined in \cite{cpc}. These various types of vertex are then grouped into a hierarchy of layers. Lemma 5 shows that any vertex in a lower layer is connected, by a single edge, to a vertex in a higher layer. 
The second part of the proof applies a curvature calculation to K. This calculation is based on the algorithm, $RSym$, described in \cite{hyp}. A quantity, equal to the number of edges less the number of vertices in K and called $curvature$, is shared equally among all the vertices of K. Proposition 6 describes how, using the layer connections established by Lemma 5, the curvature allocated to vertices can be distributed to every interior face corner of K, so that each one receives a quantity  $\geq 1/3 + 1/960n^4$. Theorem 7 shows that this implies that K satisfies a linear isoperimetric inequality which, in turn, implies that $F(2, n)$ is hyperbolic. 

\section{Classifying vertex types and arranging them into layers}
In this section we show how the symmetric nature of the relations of $F(2,n)$ implies that we can dispense with edge labelling in van Kampen diagrams and use the notion of $angles$ between edges instead. The edge colour definitions of blue, green and yellow are then introduced. Lemmas 1 and 2 prove some consequences of these definitions when colouring van Kampen diagrams. We then discuss the spherical diagram SDn and show how an arbitrary van Kampen diagram, K, can be considered \textit{spherically reduced} in that it does not contain a subdiagram matching a subdiagram in SDn which contains more than half the faces in SDn. Lemma 3 shows how being \textit{spherically reduced} limits the number of yellow edges in K. We then classify, using Lemma 4, the types of interior vertices that can occur in K and arrange these into layers. Lemma 5 then shows that any vertex in a lower layer is connected by a single edge to a vertex in a higher layer.

\subsection{Angles} 
The paper \cite{cpc} studies van Kampen diagrams of $F(2,n)$ for the values of $n$ odd and $n \geq 9$. In such diagrams, faces have the labelled form shown in Figure~\ref{triangle}.  A, B and C are vertices, each attached to a corner of the face ABC. At vertex B, the edges BA and BC with arrows pointing away from B are called \textit{outward edges} of B. The arrow on edge CA confers an order to this pair of outward edges so that BA becomes the first outward edge and BC the second outward edge of the pair. Similarly, the edges CA and CB are called \textit{inward edges} of C, where CB and CA are, respectively, the first and second inward edges of C. 

In moving from the first outward edge BA to the second outward edge BC attached to B, the edge label changes from $x_{i-1}$ to $x_{i+1}$. We can represent this label change by assigning an \textit{angle} +2 between the first and second outward edges and an angle $-2$ between the second and first outward edges. Similarly the angle between the first and second inward edges, CB and CA, of C, where the label changes from $x_{i+1}$ to $x_{i}$ is set to  $-1$, while the angle between the second and first inward edges is +1. The angle between the inward edge, AB, and the outward edge, AC, is +1, to reflect the label change from $x_{i-1}$ to $x_{i}$ We call this a \textit{transitional angle} since it involves a change from an inward to an outward edge. Similarly, the transitional angle between the outward edge AC and inward edge AB of A is given a value $-1$.

 We can assume that an arbitrary van Kampen diagram, K, is \textit{reduced} in the sense that it does not contain two identically labelled faces that share a common edge. This implies that the edges in a sequence of inward or in a sequence of outward edges, attached to a vertex in K, all have the same orientation and the same angle between each other. 
 
 When adding up the angles about an interior vertex, $v$, in a consistent direction (anticlockwise or clockwise), beginning and ending at the same edge, the \textit{transitional angles} necessarily cancel each other out. The angle sum, therefore, needs only to contain angles between consecutive edges where both edges are inward or both are outward. This angle sum must add up to $0 \mod n$.  We call this angle sum an \textit{angle sum solution} for $v$. Since $K$ is reduced, an angle sum solution cannot contain successive angles $+1-1, -1+1, +2-2$ or $-2+2$.
  
  If $v$ consists of $o$ outward edges followed by $i$ inward edges, then the number of terms in an angle sum solution is $o+i-2$. For example, the degree 5 vertex, V, shown in Figure~\ref{5v}, consists of a sequence of 3 inward edges followed by a sequence of 2 outward edges about V. The inward and outward edge sequences both have the same orientation, in the sense that the first inward edge is preceded by the last outward edge and the last inward edge is followed by the first outward edge. The angle sum solution of V consists of 3 terms and is $+1 + 1 - 2$ in the anticlockwise direction starting from and ending at the edge VD.

\begin{figure}[ht]
\centering
\includegraphics[width=8cm]{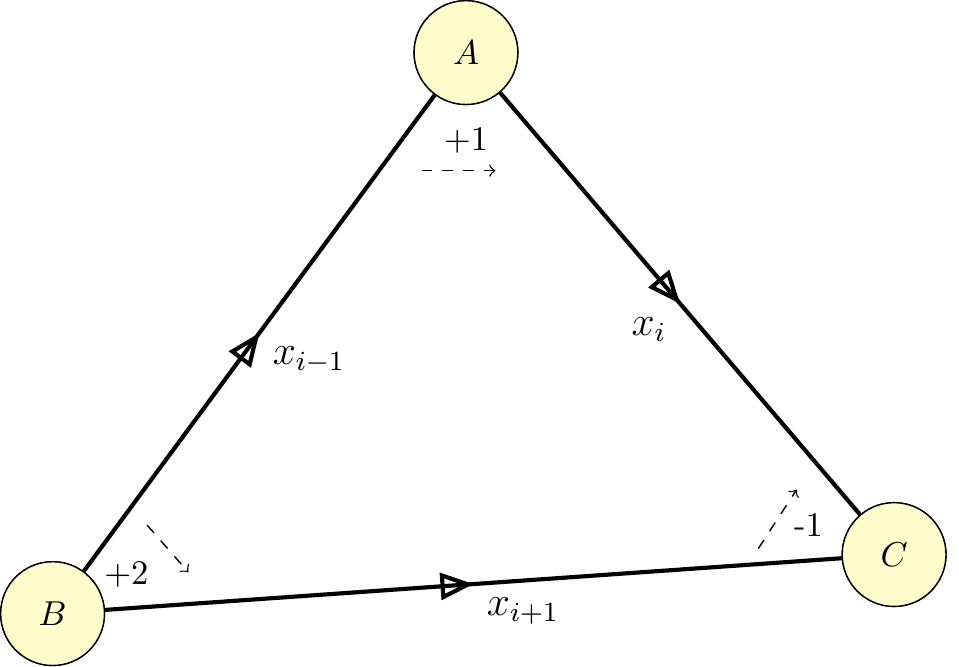}
\caption{}
\label{triangle}
\end{figure}
\newpage
\begin{figure}[ht]
\centering
\includegraphics[width=10 cm]{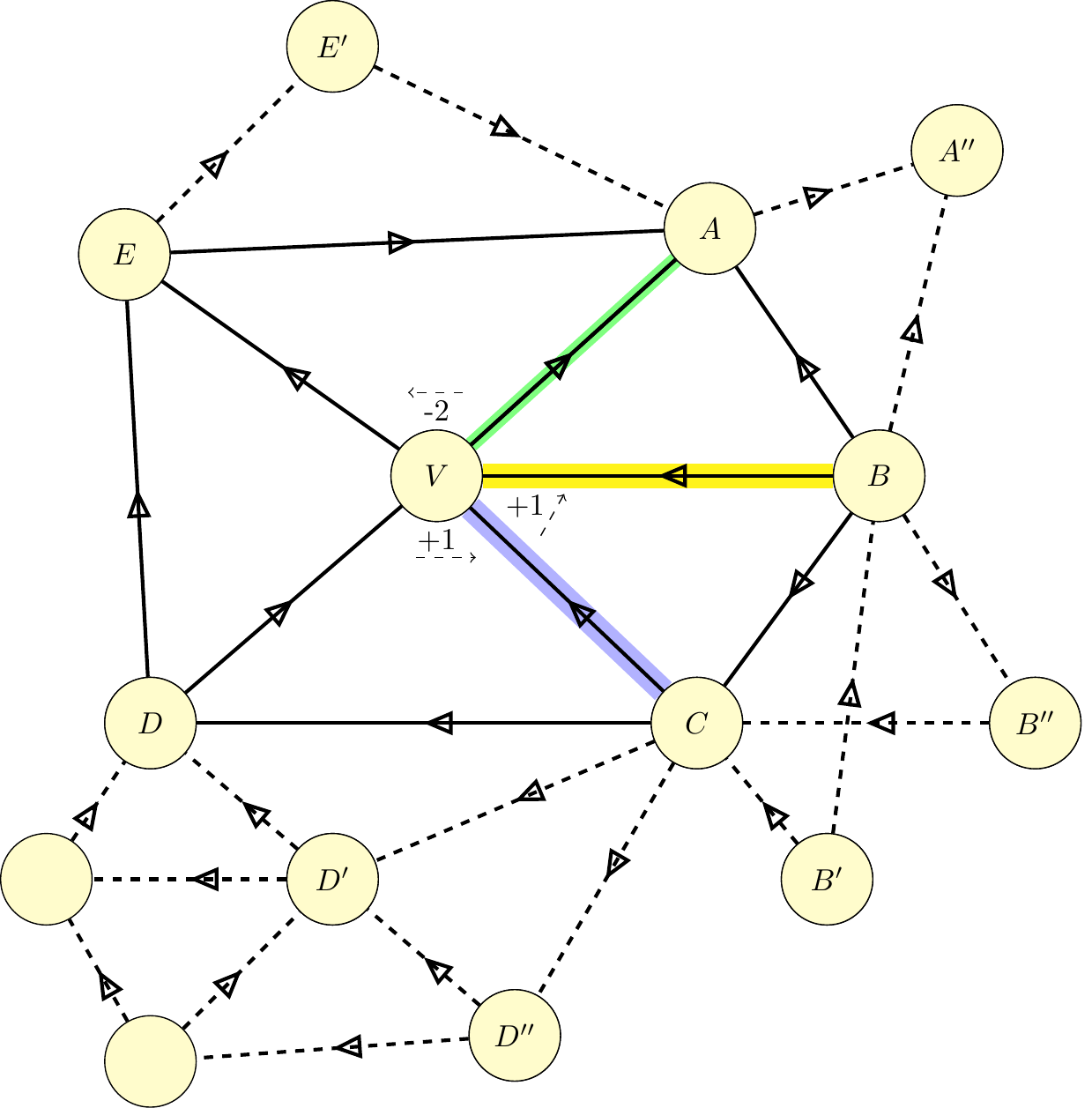}
\caption{Vertex, V, of degree 5}
\label{5v}
\end{figure}
\subsection{Coloured diagrams}
 \cite{cpc} defines 4 different colours for edges that can connect a degree 5 vertex to another vertex of a different type. Of these colours, red is reserved for the case $n$ = 9. So, for $n$ odd and $\geq 11$, we need only the three remaining colours, blue, green and yellow. With reference to Figure~\ref{5v}, these are defined as follows.
\begin{description}
\item[Blue] \hfill \\
\textit{The middle inward edge VC of V is coloured blue, if the first inward edge VB is the second outward edge in a sequence of outward edges attached to B.} 
\item[Green] \hfill \\  
\textit{The last outward edge, VA, of V is coloured green if the middle inward edge VC is not coloured blue and the vertex A is not an interior degree 5 vertex.} 
\item[Yellow] \hfill \\ 
\textit{The first inward edge, VB, of V is coloured yellow, if neither edges VA or VC are coloured green or blue, respectively, and the outward edge BV of vertex B is the middle outward edge of a sequence of 5 outward edges attached to B.} 
\end{description}
  
Lemmas 1 to 3, below, present consequences of the colour definitions that were also discussed in \cite{cpc}, Lemmas 4.1 to 4.5. 
\begin{lemma}
Every interior 5 vertex, $v$, is attached to exactly one coloured edge and this edge connects $v$ to a vertex which is not an interior 5 vertex.
\end{lemma}
\begin{proof}

We firstly show that Figure~\ref{5v} represents the only possible kind of interior vertex of degree 5 in K. Suppose $v$ is an interior vertex of degree 5. If every edge attached to $v$ is an outward edge, then the angle sum is $2+2+2+2+2$ which does not add up to 0 mod $n$ for $n \geq 11$. We can similarly rule out the case that every edge attached to $v$ is inward. So, $v$ must be attached to a mixture of inward and outward edges and the angle sum solution involves 3 terms. $2 -1 - 1 = 0$ is then the only possible angle sum solution.  So, a vertex of degree 5 must consist of 3 inward edges followed by 2 outward edges in the same orientation. 

With reference to Figure~\ref{5v}, suppose the edge VC is coloured blue. The edges BC and BV must then be the first two outward edges of a sequence of outward edges attached to B. So, either BC is a boundary edge or the edge, BB$'$, preceding BC must be an inward edge of B. Then, we observe that the inward edge sequence CB$'$ and CB of vertex C has a different orientation to the outward edge sequence CV and CD of C. So, C is not a vertex of degree 5. 
Suppose VC cannot be coloured blue. There must then be an outward edge, BB$''$, of B that precedes the outward edge BC.  
Suppose that, as well as VC being uncoloured, the edge VA cannot be coloured green. Then A must be an interior vertex of degree 5 and the inward edge AB must be the first inward edge of 3 inward edges attached to A. So, the previous edge, AA$''$, to AB must be an outward edge of A. This would mean the BA$''$ is an outward edge of B and so, the outward edge BV of B is the middle of the 5 outward edges BB$''$, BC, BV, BA and BA$''$ and so should be coloured yellow. Clearly, B is, then, not an interior vertex of degree 5. 
\end{proof}

We call an interior vertex of degree 5 a \textit{5 vertex}. 

The following lemma collects together properties of the coloured edges as they relate to vertices which are not 5 vertices. 
\begin{lemma}
Let V be a vertex which is not a 5 vertex and which is attached to a sequence of inward edges, I, followed by a sequence of outward edges, O, then 
\renewcommand{\labelenumi}{\alph{enumi}}
\begin{enumerate}
\item The first outward edge of O and the first and last edges of I are uncoloured. \item Two consecutive edges of I cannot both be coloured green. \item If an edge is blue, then it is the last outward edge of O and is followed by the last inward edge of I.  \item If O contains both a blue and yellow edge, then the blue and yellow edges are separated by at least two uncoloured edges. \item If O has at least 5 edges then at least 4 of these are uncoloured. \item If VY is the last yellow edge in O, then the following edge, VX, connects V to a 5 vertex X. 
\end{enumerate}
\end{lemma}
\begin{proof}
 We prove each part separately.
 \renewcommand{\labelenumi}{\alph{enumi}}
\begin{enumerate}
\item 
This follows from the colour definitions. A yellow edge is the middle of 5 outward edges. A green edge is the middle of 3 inward edges. A blue edge is the last in a sequence of outward edges. So, a first outward edge, a first inward edge and last inward edge are uncoloured.  
\item Without loss of generality, we can let $I$ be the inward edges AB, AV, AE and AE$'$ attached to vertex A in Figure~\ref{5v}, where AV is coloured green. Then the vertex E has two inward edges ED and EV which are not in the same orientation as the outward edges EE$'$ and EA. So, E cannot be a 5 vertex and so AE is uncoloured. \item In Figure~\ref{5v}, if VC is blue, then CV is the last outward edge of C preceded by the outward edge CD and CB is the last inward edge of C preceded by the inward edge CB$'$ \item Without loss of generality, we can let the last 4 edges of $O$ be edges CD$''$, CD$'$, CD and CV in Figure~\ref{5v}, where CV is coloured blue. If CD$'$ were a yellow edge, the vertex D$'$ would be a 5 vertex as shown. Then the two inward edges, DD$'$ and DC, and the two outward edges, DE and DV, of D have opposite orientations. D is therefore not a 5 vertex. But then the second outward edge, D$'$D of D$'$ should be green. So, CD and CD$'$ must be uncoloured edges. \item The first and second outward edges of $O$ are uncoloured by definition. If the last outward edge of $O$ is blue, then, by d, above, the previous 2 edges are uncoloured. If the last outward edge is not blue then the last two outward edges of $O$ are uncoloured. Since the number of edges of $O$ is at least 5, these edges comprise 4 uncoloured edges in $O$. \item The edge, YX, joining the 5 vertex Y to the vertex X, is the second outward edge of Y. See, for example, the edge D$'$D in Figure~\ref{5v}. If X is not a 5 vertex then, by definition, YX should be coloured green and VY uncoloured contrary to the supposition. 
\end{enumerate}
\end{proof}
\subsection{The spherical diagram, SDn}
It is observed in \cite{cpc} that there exists a spherical diagram SDn for $F(2,n)$, $n$ odd, consisting of $4n$ faces. SDn consists of two vertices of degree $n$ which are connected by outward edges to $2n$ 5 vertices. See Figure~\ref{sdn}.

For any reduced van Kampen diagram, K, with boundary label w, we can replace any subdiagram of K which exactly matches a subdiagram S of SDn, by a subdiagram which matches the complement S$'$ of S in SDn. This can be used to minimize the number of faces in K. For example, Figure~\ref{diag25} matches the subdiagram in SD11, Figure~\ref{sdn}, whose boundary is marked by the 13 vertices labelled A to M. This subdiagram contains 25 faces and 6 yellow edges but it can be replaced by a subdiagram matching its complement in SD11 that has the same 13 edge boundary and contains 19 faces and just two yellow edges. See Figure~\ref{diag19}. After a finite number of such replacements, K will contain no subdiagram matching one in SDn that has more than $2n$ faces.  When K has this property, we say that K is \textit{spherically reduced}.

\begin{figure}[hb]
\centering
\includegraphics[width=10cm]{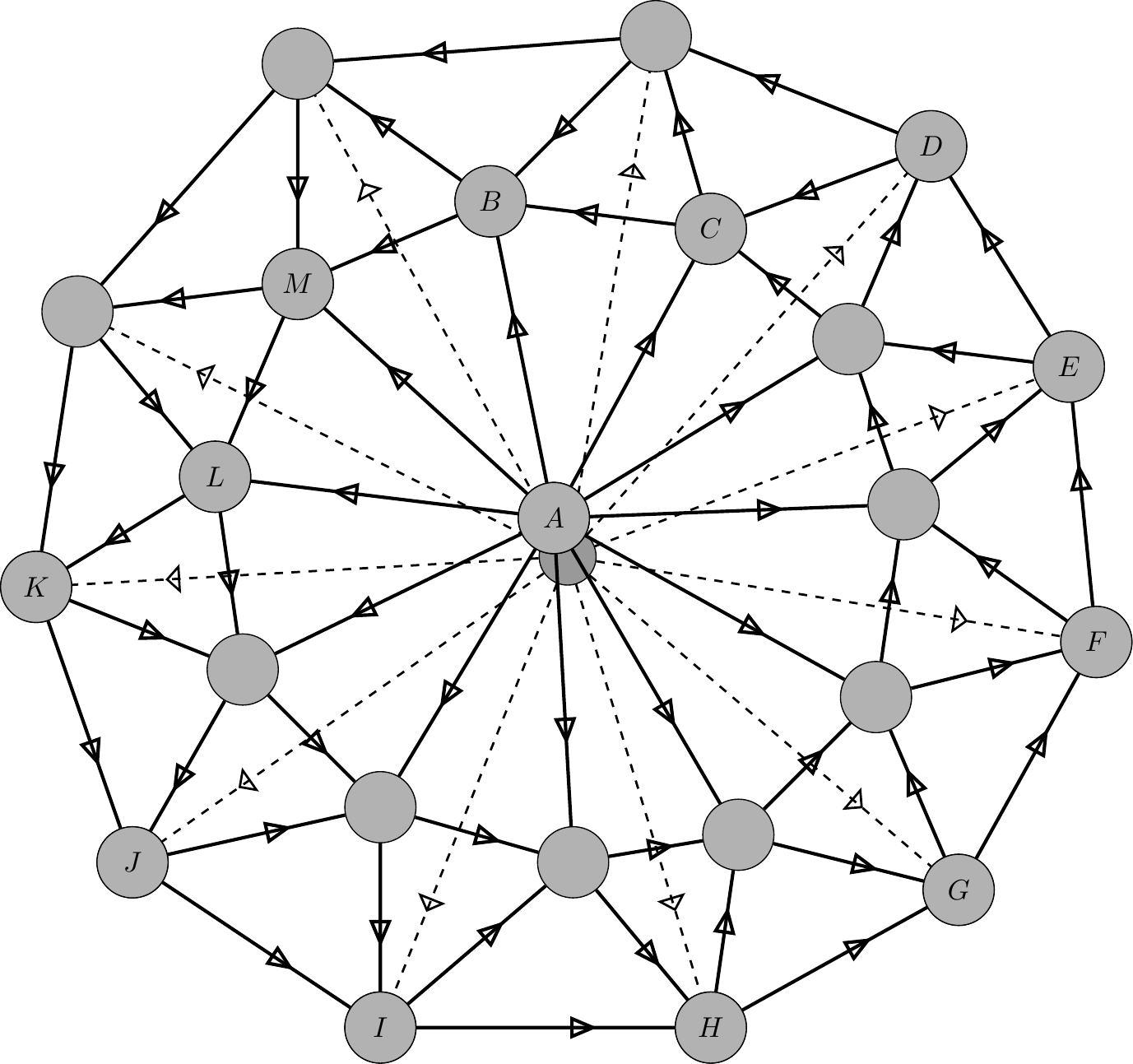}
\caption{Spherical Diagram, SD11, of F(2,11)}
\label{sdn}
\end{figure}
\newpage
\begin{figure}[h]
\centering
\includegraphics[width=6cm]{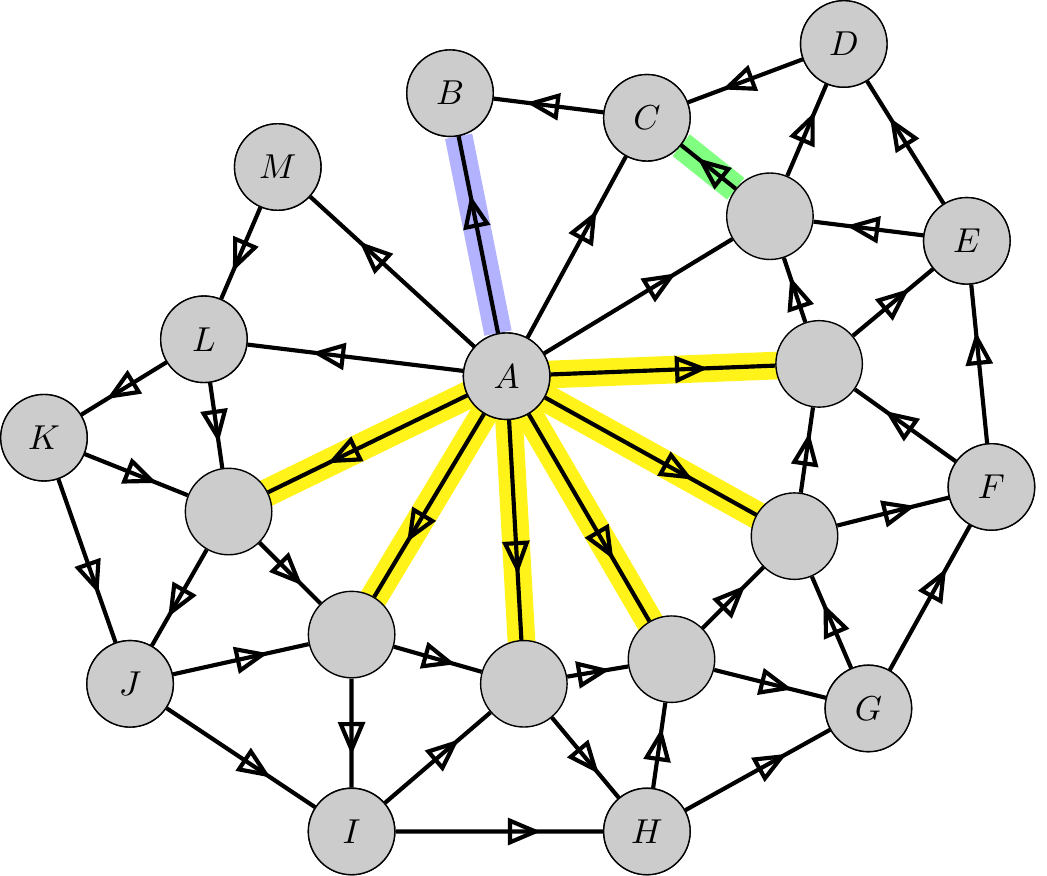}
\caption{25 face subdiagram that matches one in SD11}
\label{diag25}
\end{figure}
\begin{figure}[ht]
\centering
\includegraphics[width=6cm]{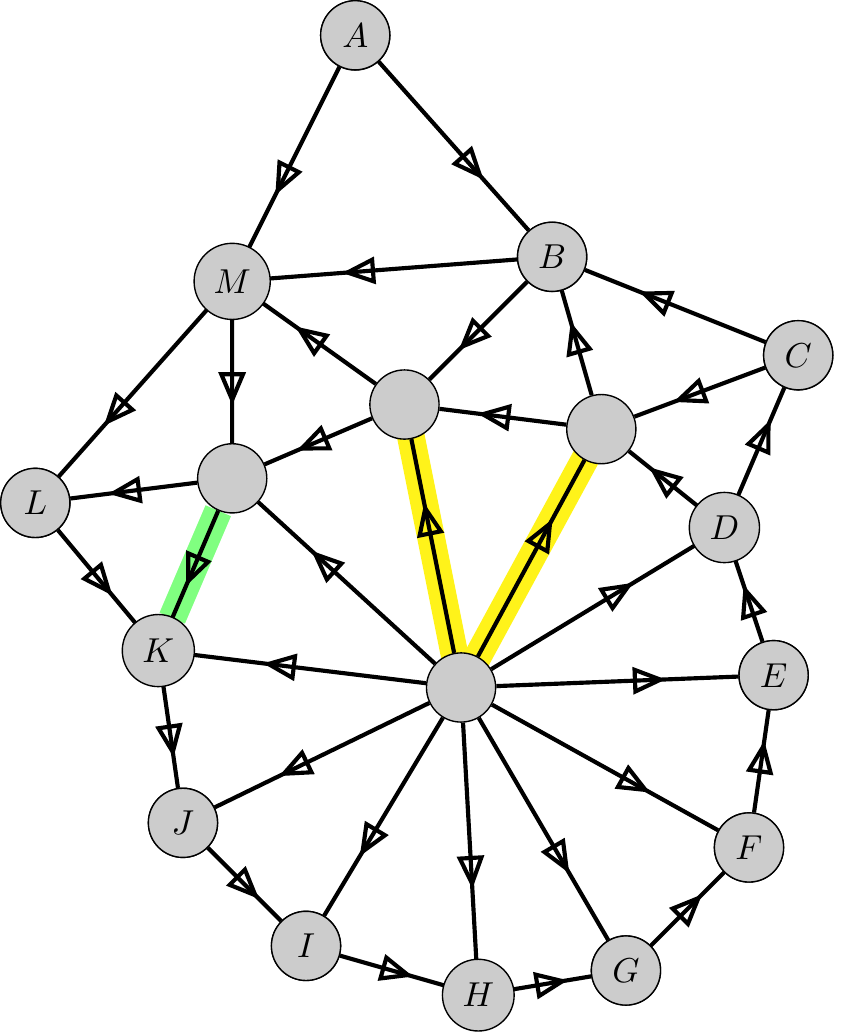}
\caption{19 face subdiagram, complement in SD11 to Figure~\ref{diag25}}
\label{diag19}
\end{figure}

\begin{lemma}
Let $K$ be spherically reduced. A consecutive sequence of $n$ outward edges attached to a vertex, $v$, contains at most $(n - 3)/2$ yellow edges. 
\end{lemma}

\begin{proof}
If every outward edge attached to $v$ is yellow, then $n$ consecutive yellow edges of $v$ together with the 5 vertices attached to these edges form a subdiagram of K, which matches a subdiagram of SDn that contains at least $3n$ faces. This contradicts $K$ being spherically reduced, so we may assume that there is a yellow edge attached to $v$ which is followed by an uncoloured edge.

Suppose there are $k$ yellow edges attached to $v$. By Lemma 2 f, an uncoloured edge that follows one of the yellow edges is attached to a 5 vertex. So, there are at least $k + 1$ 5 vertices connected to $v$ by $k +1$ outward edges. These 5 vertices form an ordered sequence that we can consider as sitting 'above' $v$ and running from left to right. The subdiagram, $S_K$, of K formed from the $n$ outward edges attached to $v$ and the attached $k+1$ 5 vertices then matches a subdiagram of SDn.  In $S_K$, a 5 vertex shares the 'bottom' 2 of its 5 faces, corresponding to the faces VAB and VAC in Figure \ref{5v}, with 2 faces attached to $v$. The 5 vertex could also share its 'top left' face, corresponding to VDC in Figure \ref{5v}, with the 'top right' face of a 5 vertex that immediately precedes it. Because K is spherically reduced, the number of  faces in $S_K$ must not be greater than $2n$. So, calculating the smallest possible number of faces in $S_K$ and taking care not to doubly count possibly shared faces, we have
\begin{equation*}
  n - 1 + 1 + 2(k+1) \leq 2n
\end{equation*}
from which we get $k \leq (n-3)/2$ since $n$ is odd.
\end{proof}
\subsection{Classifying interior vertices}
For the rest of this paper, K will indicate an arbitrary van Kampen diagram of $F(2,n)$ which is reduced, spherically reduced and has been coloured according to the colour definitions above.

An interior vertex which is attached to exactly 6 uncoloured edges is called a deg-6 vertex. An interior vertex which is attached to at least 7 uncoloured edges is called a deg-7 vertex.

We now classify the of interior vertices of $K$. 

\begin{lemma}

An interior vertex, $v$ in $K$ is either a 5 vertex, a deg-7 vertex or one of 5 types of deg-6 vertex named here as ‘GG’, ‘YB’, ‘Y+’, ‘Y-‘, and ‘5G’ and shown in the  Figures~\ref{GG} to~\ref{5G} below. 
\end{lemma}
\begin{proof}

Define a $component$ of an interior vertex, $v$, to be a sequence of outward edges followed by a sequence of inward edges.  
Let $v$ consist of $d$ components. 
\begin{description}
\item[Case $d>2$] \hfill \\
By lemma 2 a, there are 3 uncoloured edges per component, the first outward edge and the first and last inward edges. So, $v$ is attached to $\geq 9$ uncoloured edges. 
\item[Case $d=2$] \hfill \\
We can reason as above and assert that $v$ is attached to at least 6 uncoloured edges (3 uncoloured edges per component). If an outward edge sequence has more than two edges or if an inward edge sequence has more than 3 edges then, by the colour definitions and Lemma 2, $v$ will have at least one extra uncoloured edge and so at least 7 altogether. We can assume, therefore, that the maximum number of edges in an outward sequence or an inward sequence are 2 and 3 respectively. If we assign angles to the various pairs of consecutive outward or inward edges then the only combinations of 4 or 6 angles that can add up to $0 \mod n$ are the angle sums
\begin{align*}
2 - 1 - 2 + 1   \\
2 + 1 - 2 - 1  \\
2 - 1 -1 - 2 + 1 + 1   \\  
2 + 1 + 1 - 2 - 1 - 1  
\end{align*}
Figures~\ref{eight7} and~\ref{ten7}, below, show the configurations corresponding, respectively, to the angle sums $2 + 1 - 2 - 1$ (read anticlockwise) and $2 - 1 - 1 - 2 + 1 + 1$ (read clockwise).
A blue edge corresponds to the occurrence of two consecutive angles of $2 + 1$ or $- 1 - 2$ indicating that a last outward edge is followed by a last inward edge or vice versa.  We observe that only one of these angle sum pairs appear in any of the above angle sum solutions. This implies that a blue edge cannot occur in both components. So, one of the components must have at least 4 uncoloured edges and there are, in total, 7 uncoloured edges attached to $v$. 
\item[Case $d=1$] \hfill \\
If all the edges attached to $v$ are inward, then since no two successive inward edges can be coloured green, $v$ can be a deg-6 vertex only when $n=11$. This is the 5G configuration shown in Figure~\ref{5G} below.

If all the edges attached to $v$ are outward then $v$ has degree at least $n$ and, by Lemma 3, is attached to at most $(n-3)/2$ yellow edges. The number of uncoloured edges is therefore at least $(n+3)/2 \geq 7$ since $n \geq 11$.

If $v$ is attached to just 2 inward edges, then the possible angle sum solutions are 
\begin{align*}
1 + (n - 1)/2 \times 2  \\
(n + 1)/2 \times 2 - 1 
\end{align*}
These give rise to the configurations YB or Y+, of degree $(n + 3)/2 + 1$, for the first sum, and configuration Y-, of degree $(n + 5)/2 + 1$, for the second sum. See Figures~\ref{YB},~\ref{Y+} and~\ref{Y-}, below.
In each case 4 outward edges and 2 inward edges are uncoloured.

If $v$ is attached to 3 inward edges, then possible angle sum solutions are
\begin{align*}
1 + 1 - 2  \\
1 + 1 + (n-1) \times 2 
\end{align*}
The first angle sum solution is for the 5 vertex where 3 inward edges are followed by 2 outward edges and the inward and outward edges have the same orientation, as in figure 2. 
The second angle sum solution indicates that $v$ is attached to a sequence of $n$ outward edges. By Lemma 4, there are at most $(n - 3)/2$ yellow ages and there may also be a blue edge, so the number of uncoloured outward edges is at least $n - (1+(n-3)/2) = (n+1)/2 \geq 6$. Together with at least 2 uncoloured inward edges, this gives 8 uncoloured edges altogether.

If $v$ is attached to 4 inward edges, then the only angle sum solution is 
\begin{equation*}
1 + 1 + 1 + (n - 3)/2 \times 2 
\end{equation*}
But then, by Lemma 2 a, b and e, there would be 4 uncoloured outward edges and 3 uncoloured inward edges and so, $v$ would be attached to 7 uncoloured edges. 

We assume now that $v$ is attached to more than 4 inward edges and so, by Lemma 2 a and b, is attached to at least 3 uncoloured edges. 

If $v$ is attached to more than 4 outward edges, then by Lemma 2 e, $v$ would be attached to at least 4 uncoloured outward edges and so, to at least 7 uncoloured edges altogether.  

If $v$ is attached to 2 outward edges then the number of input edges must be $\geq n - 1$. Then, by Lemma 2 a and b, the number of uncoloured inward edges would be at least 6. In addition, at least one of the outward edges must be uncoloured. 

If $v$ is attached to 4 outward edges, then the only angle sum solutions are 
\begin{align*}
2 + 2 + 2 - 1 - 1 - 1 - 1 - 1 - 1    \\
2 + 2 + 2 + (1 \times (n-6)) 
\end{align*}
By Lemma 2 a and the colour definition of yellow, there are at least 3 uncoloured outward edges. 
The number of inward edges is at least 6 which by Lemma 2 a and b implies that at least 4 of them are uncoloured. 

So, we can assume that the number of outward edges is 3. 
Possible angle sum solutions for $v$ are then 
\begin{align*}
2 + 2 - 1 - 1 - 1 - 1 \\
2 + 2 + 1 \times (n - 4)
\end{align*}
The first sum gives the GG configuration of degree 8 shown in Figure~\ref{GG} below. 3 outward edges and 3 inward edges are uncoloured.
The second sum implies that there are $\geq 8$ inward edges, at least 5 of which must be uncoloured (by Lemma 2 a and b). There are also at least 2 uncoloured outward edges.

This concludes all the possibilities and we have shown that the only possible vertex types are the 5 vertex, deg-7 and the 5 deg-6 type vertices GG, YB, Y+, Y- or 5G. 
\end{description}
\end{proof}
Note that the largest degree deg-6 vertex is the 5G type which has degree $n$. 
\begin{figure}[h]
\centering
\includegraphics[width=6cm]{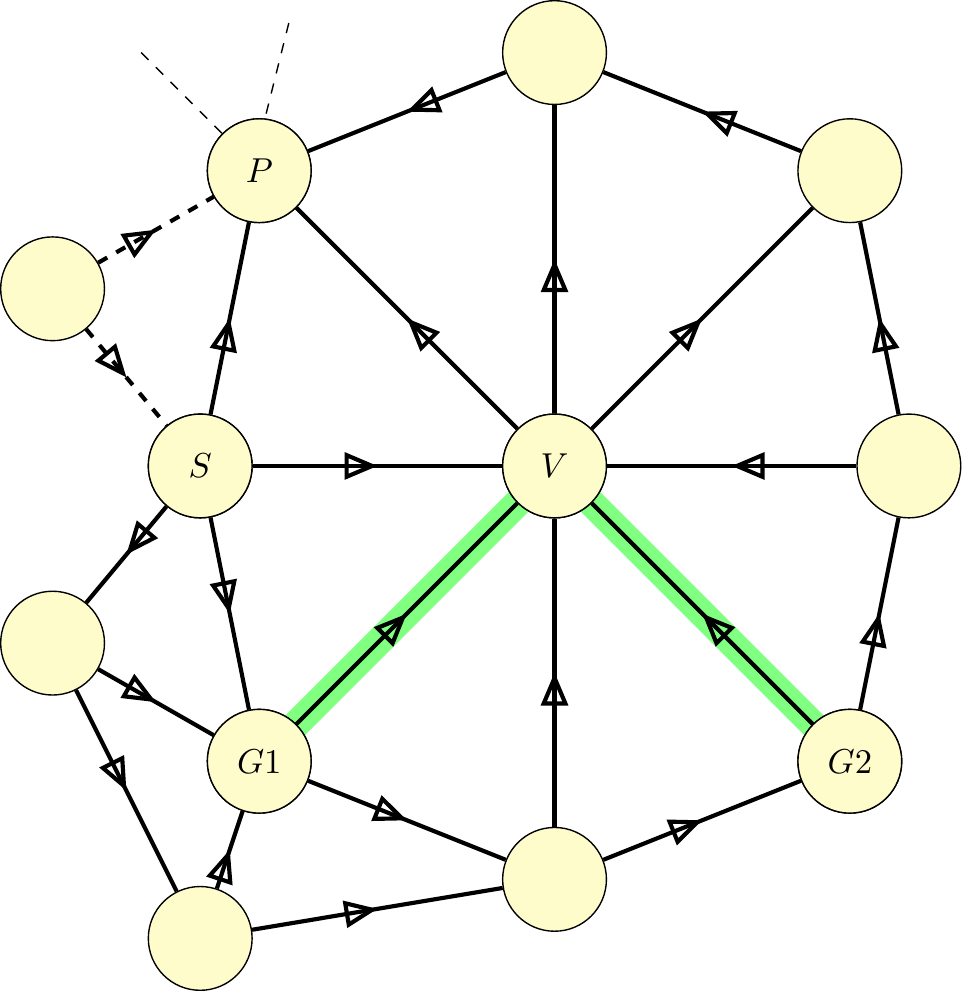}
\caption{GG deg-6 type}
\label{GG}
\end{figure}
\begin{figure}[h]
\centering
\includegraphics[width=6cm]{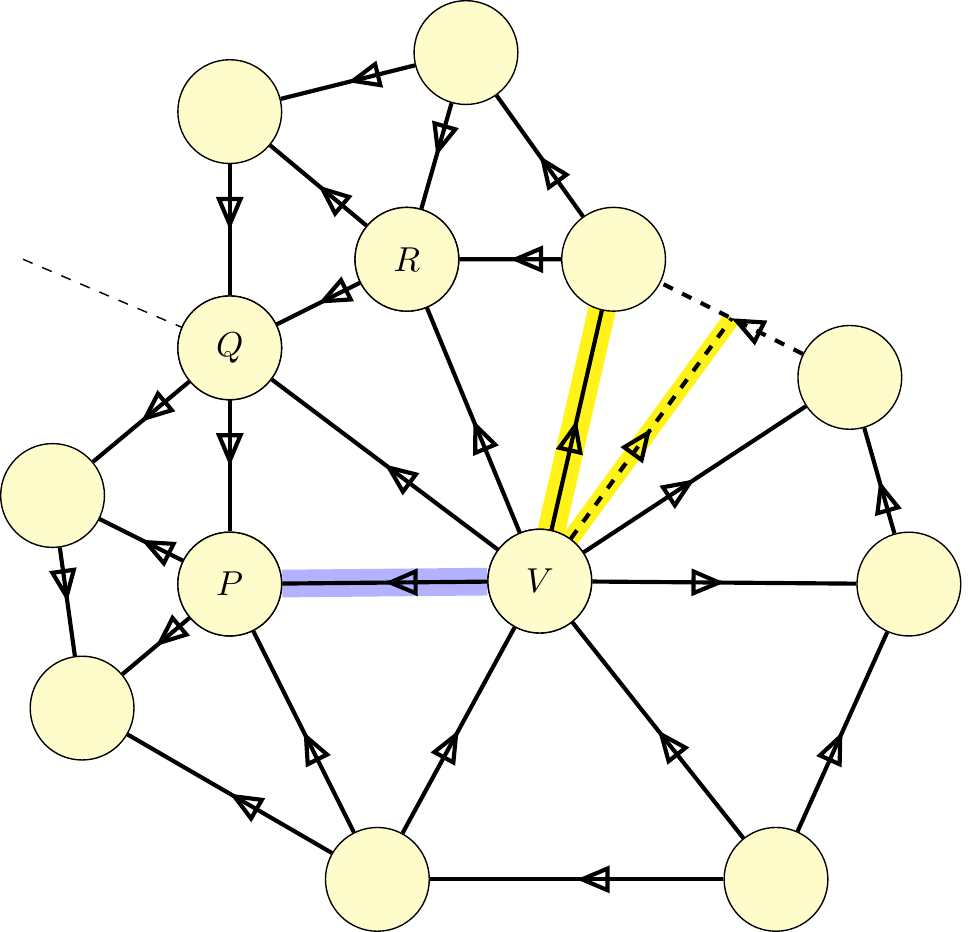}
\caption{YB deg-6 type}
\label{YB}
\end{figure}
\begin{figure}[h]
\centering
\includegraphics[width=6cm]{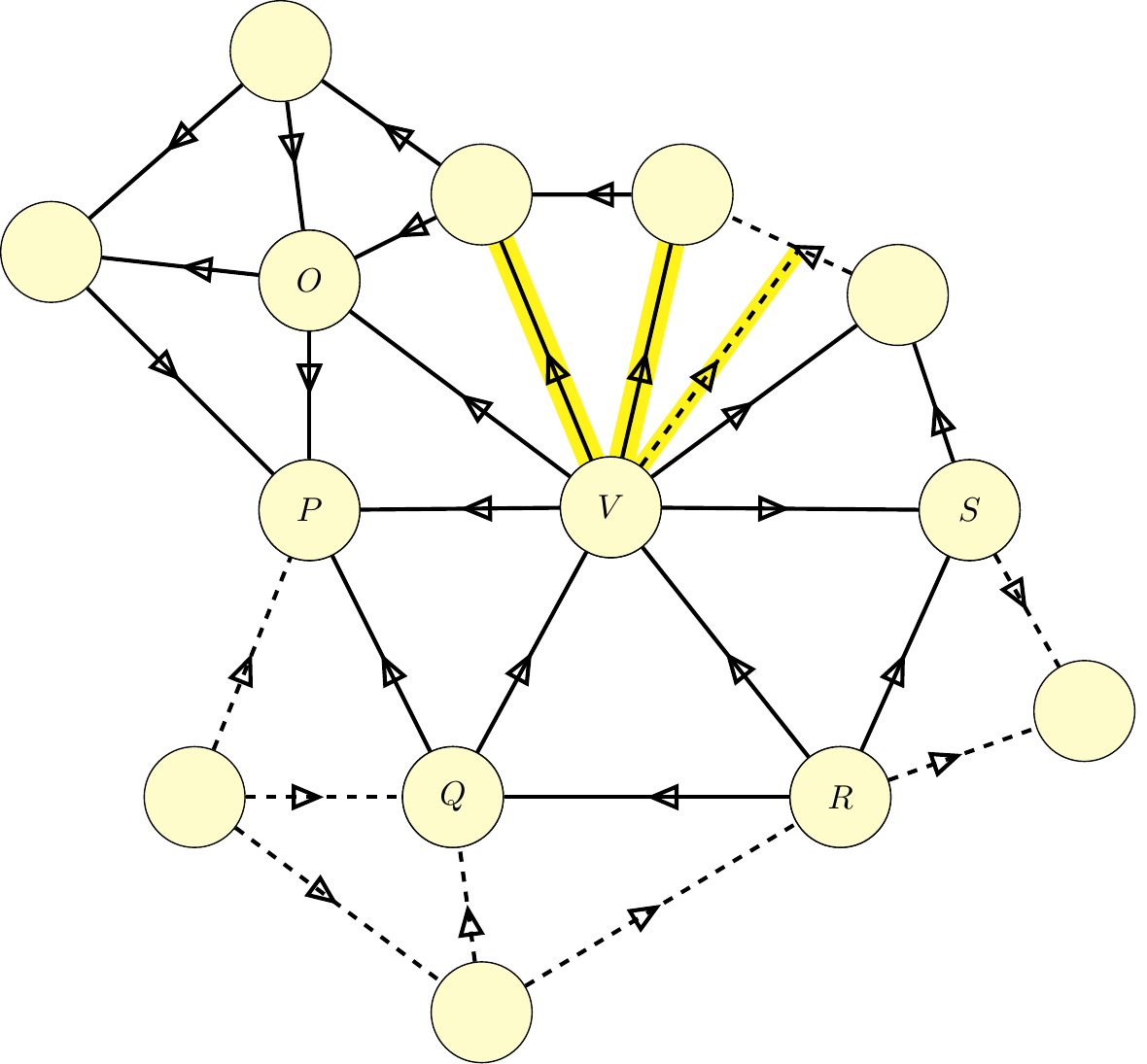}
\caption{Y+ deg-6 type}
\label{Y+}
\end{figure}
\begin{figure}[h]
\centering
\includegraphics[width=6cm]{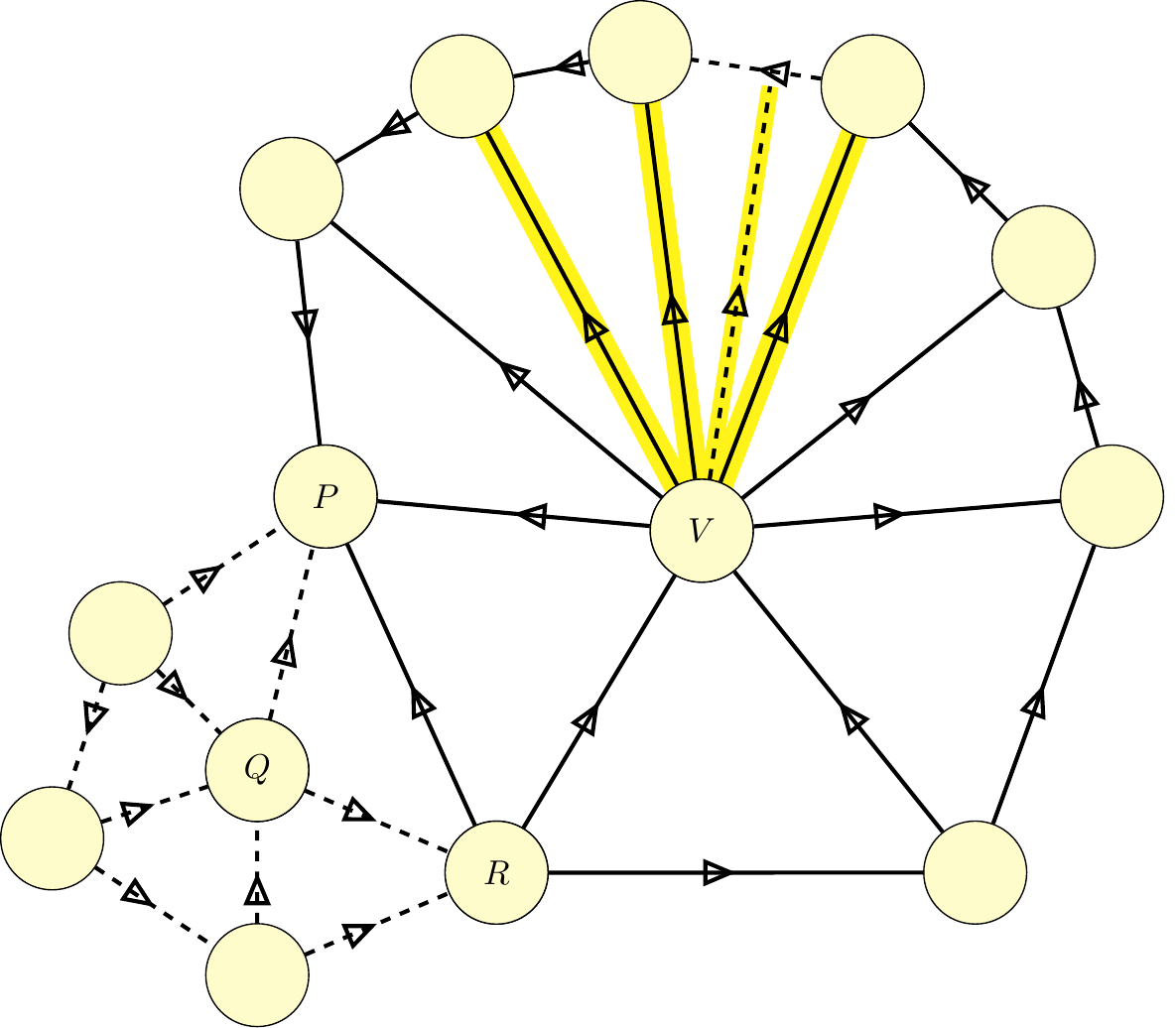}
\caption{Y- deg-6 type}
\label{Y-}
\end{figure}
\begin{figure}[h]
\centering
\includegraphics[width=6cm]{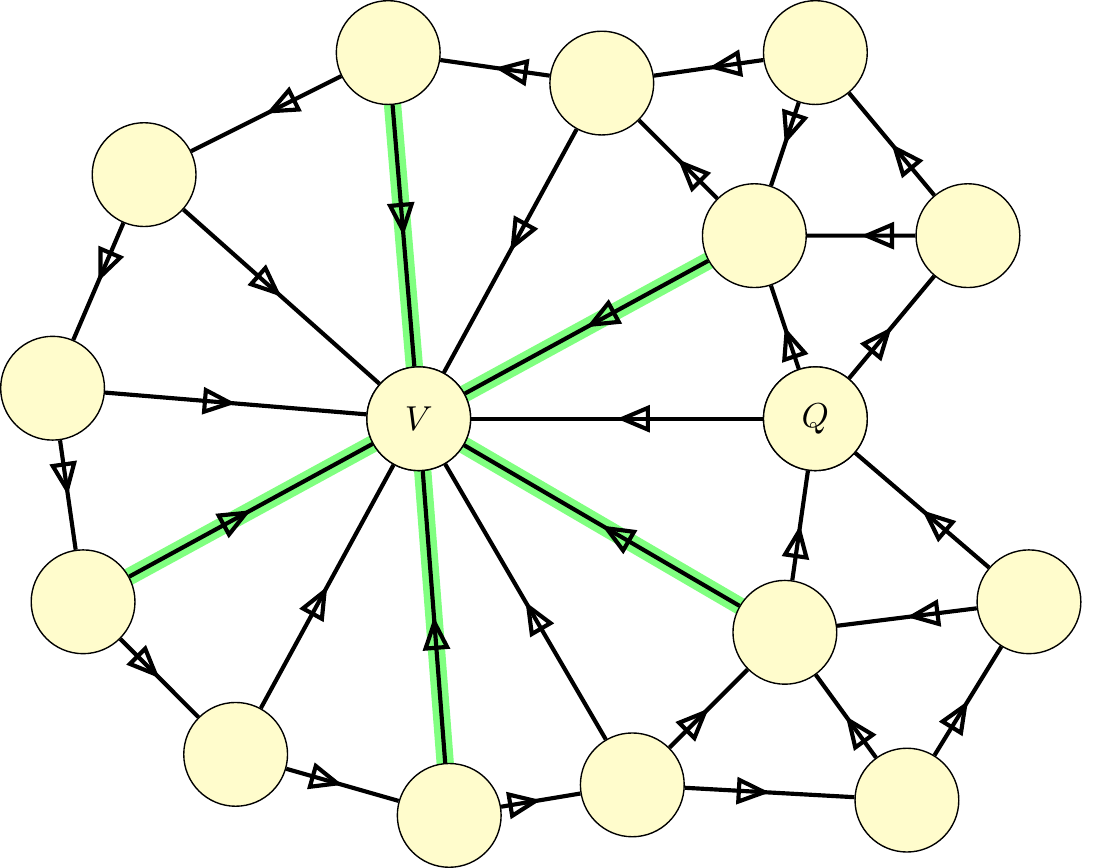}
\caption{5G deg-6 type}
\label{5G}
\end{figure}
\begin{figure}[h]
\centering
\includegraphics[width=6cm]{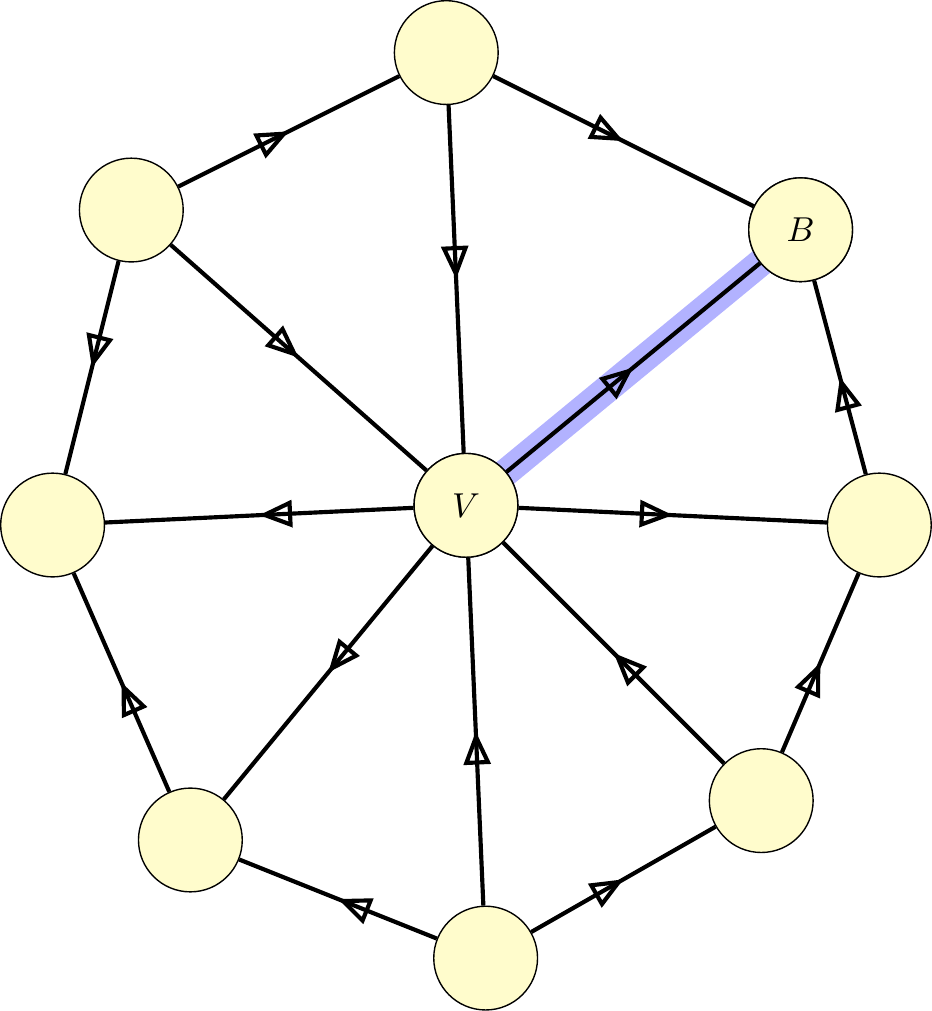}
\caption{8-7 multicomponent vertex}
\label{eight7}
\end{figure}
\begin{figure}[h]
\centering
\includegraphics[width=6cm]{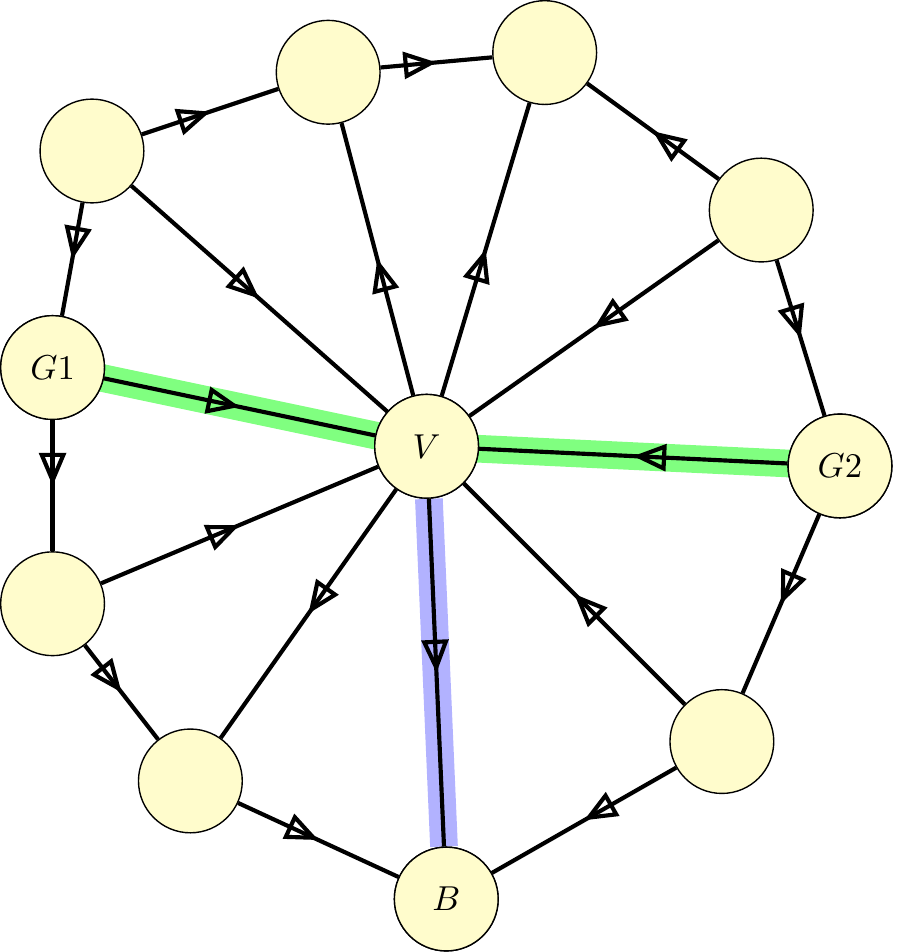}
\caption{10-7 multicomponent vertex}
\label{ten7}
\end{figure}
\newpage
\FloatBarrier
\subsection{Layers}
The vertices in K can be organised by type into a hierarchy of 5 layers, where layer 1 is the highest. \begin{description}
\item[\textit{Layer 1.   Boundary vertices and deg-7 vertices.}] 
\item[\textit{Layer 2.   GG and YB vertices.}] 
\item[\textit{Layer 3.   Y+ vertices.}] 
\item[\textit{Layer 4.   Y- and 5G vertices.}] 
\item[\textit{Layer 5.   5 vertices.}] 
\end{description}
\begin{lemma}
If a vertex, $v$, belongs to layer l, where $l > 1$ then $v$ is connected by an edge to a vertex in a higher layer m, where $m < l$.
\end{lemma}
\begin{proof}
We check this for each possible type of vertex, starting at layer 2 and ending at layer 5. 
\begin{description}
  
\item[GG case] \hfill \\ 
With reference to Figure~\ref{GG}. 
The edge, VG1, is a green edge and so vertex G1 is a 5 vertex. If S is not a boundary or deg-7 vertex, then it must be a Y+, Y- or YB type vertex. But then, because neither SV and SG1 are yellow edges (since a 5 vertex can only be attached to one coloured edge) and SG1 is not the first outward edge of S, S must be a YB type vertex such that SP is the last outward edge, SP is a blue edge and P has degree 5. But this is impossible because then P would be attached to 4 inward edges.  
\item[YB case] \hfill \\
With reference to Figure~\ref{YB}. 
By Lemma 2 f and the definition of blue, R and P are 5 vertices. This forces Q to be attached to 3 consecutive inward edges, starting at QV, which are oppositely oriented to the two outward edges starting at QP. This does not match any of the 5 types of deg-6 vertices established in Lemma 4, so, Q must be a deg-7 or boundary vertex. 
\item[Y+ case] \hfill \\ 
With reference to Figure~\ref{Y+}. 
Assume no vertex attached to $V$ is a deg-7 or a boundary vertex. Vertex O is a 5 vertex by Lemma 2 f, so P is either a GG or a 5G vertex. Suppose P is a 5G vertex. Q is then attached to exactly 2 outward edges so is not a deg-6 type vertex.  Q must therefore be a 5 vertex as shown.  Assume R is a YB, Y- or Y+ type vertex where RS is the third outward edge. But then RS would be yellow and S would be a 5 vertex with an inward sequence of only 2 edges, which is impossible.  So, R must be a GG vertex type. 
\item[Y- case] \hfill \\ 
With reference to Figure~\ref{Y-}. 
We argue as in the Y+ case and assume that no vertex attached to $v$ is a deg-7 or boundary vertex. We assume that P is not a GG vertex but a 5G vertex. Figure 11 shows that the 5G vertex type is attached to at least one green edge for every 3 consecutive inward edges. Vertices R and V are not 5 vertices so PR and PV are uncoloured. PQ must, therefore, be green and Q must be a 5 vertex as shown. Then, we see that R is attached to two inward edges in opposite orientation to at least 3 outward edges attached to R and so R must be a Y+ type vertex. 
\item[5G Case] \hfill \\ 
With reference to Figure~\ref{5G}. 
The vertex Q is attached to V and two adjacent 5 vertices. We then see that, if Q is not a boundary or deg-7 vertex, it must be a Y+ type vertex. 
\item[5 vertex] \hfill \\ 
If $v$ is a 5 vertex then, by lemma 1, it is attached by a coloured edge to a vertex which is not a 5 vertex. 
\end{description}
\end{proof}

\section{Applying curvature calculations to the \\ van Kampen diagram, K, to prove hyperbolicity.}
Our curvature calculations are based on the RSym algorithm defined in \cite{hyp} section 6. We suppose that K has boundary label $w$ and consists of $F$ faces, $E$ edges and $V$ vertices. The \textit{Area} of K is the number of faces, $F$, in K. The \textit{Length} of K is the word length of $w$. The faces of K which share one or more boundary edges with the boundary of K are called boundary faces.  All other faces are called interior faces and their corners are called interior face corners. If there are $B$ boundary faces, then the number of interior faces is $F - B$. Clearly, \textit{Length} $\geq B$. If $v$ is a vertex of K then $d(v)$ denotes the degree of $v$. 

The RSym algorithm assigns +1 curvature to faces and vertices and -1 curvature to edges. Negative curvature is distributed firstly from edges to vertices and then from vertices to interior faces. Hyperbolicity of the associated group then follows if each interior face ends up with a   negative curvature $\leq$ a constant. We adapt the RSym algorithm in the following way. We call the quantity  $E - V$, $curvature$, and we share it out to all the vertices of K. Because $E = \sum_{v \epsilon V} d(v)/2$, $E - V$ can be shared out so that every vertex, $v$, of K receives, initially, a $curvature$ of $d(v)/2 - 1$.  Curvature is then distributed from each vertex to its attached interior face corners and to neighbouring vertices.  Proposition 6 shows that the distribution of curvature can be organised so that each interior face corner ends up with at least $1/3 + 1/960n^4$ curvature. Theorem 7 shows that this implies that the associated group $F(2,n)$ is hyperbolic.
\begin{proposition}
The curvature quantity, $E-V$, can be distributed so that every interior face corner in K receives a curvature amount of at least $1/3 + 1/960n^4$. 
\end{proposition}
\begin{proof}
 
 We find it convenient to simplify the shape of boundary vertices in K, where necessary, so that we can assume that each one is attached to exactly two boundary edges.
 We do this by performing two types of change on K to produce a possibly fragmented diagram K*. 
 \begin{enumerate}
     \item We remove any boundary edges and boundary vertices which are not attached to a face. These edges form a collection of trees, where each tree has at least one vertex attached to a face. So the number of edges removed is at least the number of vertices removed. After this change, every boundary vertex is left attached to one or more pairs of consecutive boundary edges, where each boundary edge pair encloses one or more faces.
     \item
     If a boundary vertex, $v$, is attached to $p > 1$ boundary edge pairs, we replace $v$ by a collection, $P$=\{$v_i,i=0,..,p-1$\}, of $p$ vertices.  Each new vertex, $v_i$, is allocated a unique pair of consecutive boundary edges, ($e_{2i},e_{2i+1}$), that were formerly attached to $v$ and which enclose one or more boundary faces and possibly one or more interior edges. The vertex $v_i$ becomes the new common endpoint of $e_{2i}$ and $e_{2i+1}$ and also the common end point of any interior edges that lie between $e_{2i}$ and $e_{2i+1}$ and that were once attached to $v$. No edges have been removed or added, so it's clear that the total number of interior face corners attached to the vertices in $P$ is the same as the number of interior face corners attached to $v$.  
 \end{enumerate}
 Every boundary vertex in K* is then attached to exactly 2 boundary edges. All interior vertices in K and boundary vertices not affected by change 2 are left unaltered, so the number of interior face corners in K* is the same as the number of interior face corners in K.
 
 We let $E*$ and $V*$ be the number of edges and vertices, respectively, in K*. The changes made to K imply that the curvature, $E*$ $-$ $V*$, of K*  is less than or equal to the curvature, $E - V$, of K.  We prove the proposition by showing it is true for K* and deduce that it must then also be true for K. The curvature, $E*$ $-$ $V*$, is distributed equally to the vertices of K*. As explained above, each vertex, $v$, in K* then has, initially, a curvature value of $d(v)/2 - 1$.  

 In general, we consider each vertex, $v$, in K* to be comprised of one or more disjoint components. Each component, $c$, of $v$ consists of a number, $m_c$, of consecutive edges attached to $v$. If $C$ is the collection of components of $v$, then the initial curvature of $d(v)/2 - 1$ assigned to $v$ is shared out to each component $c$ as $m_c/2 - f_c$ where $\sum_{c \epsilon C} m_c = d(v)$ and $\sum_{c \epsilon C} f_c = 1$.
 We show that the proposition is true for each type of vertex in the classification established by Lemma 4. We start with the boundary vertices and deg-7 vertices in the top layer, layer 1, of vertices and then work down through the layers.
 
 A vertex in K* is called \textit{large} if it has degree $> n$ and is attached to either at least $n$ consecutive outward edges or at least $n$ consecutive inward edges. By Lemma 4, a $large$ vertex is either a boundary vertex or a deg-7 type vertex. We deal with \textit{large} vertices first so that in subsequent cases we can assume that no vertex is attached to more than $n$ consecutive inward or outward edges.
\begin{description}
\item[Layer 1. $Large$ vertices.] \hfill \\ 
Let $v$ be a \textit{large} vertex. We split the edges attached to $v$ into \textit{big and small components} as follows. Each component type consists of at most $n$ consecutive outward or at most $n$ consecutive inward edges. If a sequence of edges attached to $v$ consists of $m > n$ outward or inward edges, then $[m/n]$ $big$ components, each containing $n$ consecutive edges, are assembled by starting from the last outward or inward edge and working backwards $n$ edges at a time. Any remaining consecutive sequences of outward or inward edges, each one containing less than $n$ edges,  form a $small$ component. Because a $small$ component always contains the start edge of an inward or outward edge sequence, it is guaranteed, by Lemma 2 a, to contain at least one uncoloured edge.  The -1 quantity of the initial curvature of $d(v)/2 - 1$, allocated to $v$, is shared amongst the components as follows.
\begin{itemize}
    \item  -11/12 to an arbitrarily chosen $big$ component,
    \item -1/12 to another arbitrarily chosen $big$ or $small$ component,
    \item 0 to all other remaining $big$ or $small$ components
\end{itemize}

We show that, for each $big$ or $small$ component of $v$, we can share a curvature of
\begin{itemize} 
\item $1/3 + 1/24n$ to each attached interior face corner,
\item 1/6 to each 5 vertex attached by a coloured edge,
\item $1/24n$ to each attached interior vertex
\end{itemize}
Let c be a $big$ component of $n$ outward edges. Then, by Lemma 3, c contains at most $(n-3)/2$ yellow edges. It could also contain a blue edge as the last edge of the outward edge sequence. So altogether it may contain $(n - 1)/2$ coloured edges. If, instead, c is a $big$ component of $n$ inward edges, then, by Lemma 2 b, it contains, also, at most $(n - 1)/2$ coloured edges. A $big$ component contains at most $n$ interior face corners, up to the start of the next component, and has at most $n$ connected interior vertices.

So, for an arbitrary $big$ component, the total curvature of the component less the curvature to be shared is at least
\begin{align*}
\frac{n}{2} - \frac{11}{12} - (\frac{n}{3} + \frac{n}{24n} + \frac{n-1}{2} \times \frac{1}{6}+ \frac{n}{24n}) =
\frac{n-11}{12} \geq 0     
\end{align*}
since $n \geq 11$.

Let c be a $small$ component containing $k$ edges where $k < n$. As observed above, it has at least one uncoloured edge, is attached to at most $k$ interior corners, up to the starting edge of the next component, and connected to at most $k$ interior vertices
So, for an arbitrary $small$ component, the total curvature of the component less the curvature to be shared is at least
\begin{align*}
\frac{k}{2} - \frac{1}{12} - (\frac{k}{3} + \frac{k}{24n} + \frac{k-1}{6} + \frac{k}{24n}) =
\frac{1}{12} - \frac{k}{12n} > 0     
\end{align*}
since $k<$ $n$

\item[Layer 1. Boundary vertices.] \hfill \\ 
Recall that every boundary vertex in K* is attached to exactly 2 boundary edges.
Let $v$ be a boundary vertex. We can assume that $d(v) > 2$,  otherwise, there are no interior face corners or interior vertices connected to $v$. We define a component of $v$ to be a maximal sequence of 2 or more inward edges or a maximal sequence of 2 or more outward edges. If either one of the boundary edges don’t belong to a component, we append it to the first or last component of $v$ as appropriate. If the first or last component consists of only two outward edges and the boundary edge is the first outward edge of the sequence, then we merge this first or last component with the following or previous component as appropriate. We do this to be sure that when the degree of $v$ is $> 3$, there is at least one interior uncoloured edge per component. For, if all of the interior edges of a component are inward then, by Lemma 2 b, at least one is uncoloured. If all of the interior edges are outward, there are at least two such edges and so the first of these edges is either a first or second outward edge and not last and is, therefore, uncoloured by the definition of yellow. If the interior edge set consists of a mixture of inward and outward edges then there is either a first or last inward edge and this is uncoloured by Lemma 2 a. 

The number of edges in a component is $\leq n + 2$ which, in turn, is clearly $ < 2n$.

We show that for each component we can share a curvature of
\begin{itemize} 
\item $1/3 + 1/24n$ to each attached interior face corner,
\item 1/6 to each 5 vertex attached by a coloured edge,
\item $1/24n$ to each attached interior vertex
\end{itemize}
We do this by calculating that the total curvature of the vertex is at least that of the curvature to be shared. 
\item[Case $v$ has one component] \hfill \\
Suppose that the degree of $v$ is 3. There are no interior face corners, at most one coloured edge and at most one attached interior vertex.
Then the total curvature of the vertex less the curvature to be shared is
\begin{align*}
 3/2 - 1 - (1/6 + 1/24n) =
1/3 - 1/24n > 0   
\end{align*}

If the degree of $v$ is $k>3$ then $v$ is attached to $k - 2 > 1$ interior edges.  The component has at least one uncoloured interior edge and, since the boundary edges are also uncoloured, there are at least 3 uncoloured edges and so at most $k - 3$ coloured edges attached to $v$. There are $k - 1$ face corners attached to $v$, but only $k - 3$ of them are interior face corners.
So the total curvature of the vertex less the curvature to be shared is at least
\begin{align*}
k/2 - 1 - (\frac{k-3}{3} + \frac{k-3}{24n}+ \frac{k-3}{6} + \frac{k-2}{24n}) =
1/2 - \frac{2k - 5}{24n} > 0      
\end{align*}
since $k < 2n$,

\item[Case $v$ has more than one component] \hfill \\
We share the -1 part of the initial curvature, $d(v)/2 - 1$, of $v$ as -1/2 each to the first and last components. Let the first component have $k$ edges. There are, then, $k - 1$ interior face corners up to the start of the next component, at most $k - 2$ coloured edges and at most $k - 1$ attached interior vertices.
Then the total curvature of the component less the curvature to be shared is at least
\begin{align*}
k/2 - 1/2 - (\frac{k-1}{3} + \frac{k-1}{24n} + \frac{k-2}{6} + \frac{k-1}{24n}) =
1/6 - \frac{2k-2}{24n} > 0     
\end{align*}
since $k < 2n$.

A similar calculation applies to the last component comprising $l$ edges say. Then there are only $l - 2$ interior face corners and so the total curvature less curvature to be shared amount would be 
\begin{align*}
1/2 - \frac{2l-3}{24n} > 0   
\end{align*}
since $l < 2n$.

Let a ‘middle’ component have $m$ edges then, there are $m$ attached interior face corners up to the start of the next component, at most $m - 1$ coloured edges and at most $m$ attached interior vertices.
Then the total curvature of the component less the curvature to be shared is at least
\begin{align*}
m/2 - (m/3 + \frac{m}{24n}+ \frac{m-1}{6} + \frac{m}{24n}) =
1/6 - \frac{2m}{24n} > 0    
\end{align*}
since $m < 2n$
\item[Layer 1. deg-7 vertices.] \hfill \\
We want to show, as in the boundary vertex case, that for any deg-7 vertex we can share a curvature of
\begin{itemize} 
\item $1/3 + 1/24n$ to each interior face corner,
\item 1/6 to each 5 vertex attached by a coloured edge,
\item $1/24n$ to each attached interior vertex
\end{itemize}
Let $v$ be a deg-7 vertex and so attached to at least 7 uncoloured edges. This time, we take a component to be a sequence of 2 or more inward edges followed be a sequence of 2 or more outward edges and assume, firstly, that $v$ is made up of $d$ components where $d > 1$.
\item[Case d $> 2$.] \hfill \\
We share the -1 part of the initial curvature of $v$ as $-1/d$ per component. Let there be $k$ edges in a particular component. There are $k$ attached interior face corners up to the next component. By Lemma 2 a, at least 3 edges are uncoloured in a component and so there are at most
$k - 3$ coloured edges. There are at most $k$ attached interior vertices. Then, per component, the amount of curvature of the component less the curvature to be shared is at least
\begin{align*}
\frac{k}{2} - \frac{1}{d} - (\frac{k}{3} + \frac{k}{24n} + \frac{k-3}{6} + \frac{k}{24n}) =
\frac{1}{2} - \frac{1}{d} - \frac{2k}{24n} \geq \frac{1}{6} - \frac{2k}{24n} \geq 0     
\end{align*}
since $d > 2$ and $k \leq 2n$.
\item[Case $d=2$.] \hfill \\
Let $k$ and $m$ be the number of edges of component 1 and component 2 respectively. Let uc1 be the number of uncoloured edges for component 1 and uc2 be the number of uncoloured edges for component 2.  Suppose uc1 is 3, then uc2 cannot also be 3 since $v$ is not one of the deg-6 vertex types. If uc1 is 3 and uc2 is 4, the only possibilities for $n \geq 11$ are  Figures~\ref{eight7} and~\ref{ten7} shown above. We can treat these as single component deg-7 vertices, where we define the component to be the whole vertex, and this case is dealt with in the ‘Case $d$ = 1’ below.
If uc1 is 3 and uc2 $ > $4, then the -1 part of the initial curvature of $v$ is shared out as -2/3 for component 2 and -1/3 for component 1. Then, per component, the amount of curvature available less the curvature to be shared is at least
\begin{align*}
 k/2 - 1/3 - (k/3 + \frac{k}{24n} + \frac{k-3}{6} + \frac{k}{24n}) =
1/6 - \frac{2k}{24n} \geq 0  
\end{align*}
since $k \leq 2n$, and
\begin{align*}
m/2 - 2/3 - (m/3 + \frac{m}{24n} + \frac{m-5}{6} + \frac{m}{24n}) =
1/6 - \frac{2m}{24n} \geq 0 
\end{align*}
since $m \leq 2n$.

If uc1 and uc2 are both $\geq 4$, the -1 part of the initial  curvature of $v$ is shared out as -1/2 for each component.  Then, for the first component, the amount of  curvature available less the  curvature to be shared is
\begin{align*}
k/2 - 1/2 - (k/3 + \frac{k}{24n} + \frac{k-4}{6} + \frac{k}{24n}) =
1/6 - \frac{2k}{24n} \geq 0 
\end{align*}
since $k \leq 2n$. 
And, similarly, for the second component
\item[Case $d = 1$.] \hfill \\
We now define component to be simply the whole vertex. Suppose the vertex has degree $k$. The amount of curvature available less the curvature to be shared is
\begin{align*}
k/2 - 1 - (k/3 + \frac{k}{24n} + \frac{k-7}{6}+ \frac{k}{24n}) =
1/6 - \frac{2k}{24n} \geq 0    
\end{align*}
since $k \leq 2n$
\item[Layer 2. GG, YB vertices] \hfill \\
Let $v$ be a GG or YB vertex of degree $k$. By Lemma 5, these types of vertices are attached by at least one edge to a boundary or deg-7 vertex. So, these vertices will receive at least $1/24n$ curvature from one or more of their attached vertices. This time, we want to show that $v$ has enough curvature so that we can share
\begin{itemize} 
\item $1/3 + 1/48n^2$ to each interior face corner,
\item 1/6 to each 5 vertex attached by a coloured edge,
\item $1/48n^2$ to each attached interior vertex
\end{itemize}
There are $k - 6$ coloured edges, at most $k$ interior face corners and at most $k$ interior vertices attached to $v$.  The amount of curvature available less the amount needed to be shared is therefore
\begin{align*}
\frac{1}{24n} + k/2 - 1 - (k/3 + \frac{k}{48n^2} + \frac{k-6}{6} + \frac{k}{48n^2}) =
\frac{1}{24n} - \frac{2k}{48n^2} > 0   
\end{align*}
since $k < n$. 
\item[Layer 3. Y+ vertices.] \hfill \\
Let $v$ be a Y+ vertex of degree $k$. By Lemma 5, $v$ is attached to either a GG, deg-7 or a boundary vertex. So, $v$ will receive at least $1/48n^2$  curvature from one or more of its attached vertices. This time we want to show that $v$ has enough curvature to share at least
\begin{itemize}
\item $1/3 + 1/96n^3$ to each interior face corner,
\item 1/6 curvature to each 5 vertex attached by a coloured edge
\item $1/96n^3$ to each attached interior vertex
\end{itemize}
As in the previous case, there are $k - 6$ coloured edges, at most $k$ interior face corners and at most $k$ interior vertices attached to $v$.  The amount of  curvature available less the amount needed to be shared is therefore at least
\begin{align*}
\frac{1}{48n^2} + k/2 - 1 - (k/3 + \frac{k}{96n^3} + \frac{k-6}{6} + \frac{k}{96n^3}) =
\frac{1}{48n^2} - \frac{2k}{96n^3} > 0    
\end{align*}
since $k < n$
\item[Layer 4. Y-, 5G vertices.] \hfill \\
Let $v$ be a Y- or 5G vertex of degree $k$. By Lemma 5, $v$ is attached to at least one vertex from a higher layer and so $v$ will receive at least $1/96n^3$  curvature from its attached vertices. This time we want to show that $v$ has enough curvature to share at least
\begin{itemize}
\item $1/3 + 1/192n^4$ to each interior face corner,
\item $1/6 + 1/192n^4$ curvature to each 5 vertex attached by a coloured edge
\end{itemize}
There are $k - 6$ coloured edges and at most $k$ interior face corners attached to $v$.  The amount of curvature available less the amount needed to be shared is therefore
\begin{align*}
\frac{1}{96n^3} + \frac{k}{2} - 1 - (\frac{k}{3} + \frac{k}{192n^4} + \frac{k-6}{6} + \frac{k-6}{192n^4}) =
\frac{1}{96n^3} - \frac{2k - 6}{192n^4} > 0 
\end{align*}
since $k \leq n$
\item[Layer 5. 5 vertices.] \hfill \\
By Lemma 1, a 5 vertex is attached by a coloured edge to a vertex from a higher layer. It therefore receives curvature of at least $1/6 + 1/192n^4$ from this vertex. Its curvature is, therefore, at least
\begin{align*}
 5/2 - 1 + 1/6 + 1/192n^4 = 5/3 + 5/960n^4  
\end{align*}
It is able, therefore to share this curvature to at most 5 attached interior face corners and give at least $1/3 + 1/960n^4$ to each corner.
\end{description}
All possible vertex types have been covered, and we have shown that the curvature of $E*$ $-$ $V*$ can be distributed to give each interior face corner in K* at least $1/3 + 1/960n^4$ of curvature. 

Since $E*$ $-$ $V*$ $\leq E-V$ and the number of interior face corners in K* and K are the same, the above statement is also true for the curvature $E-V$ and the interior face corners of K. 
\end{proof}

\begin{theorem}
$F(2, n)$ is hyperbolic for $n$ odd and $n \geq 11$.
\begin{proof}
Let K, $V, E, F, B, Area$ and $Length$ be defined as at the start of  section 3, above.  A total of at most $E - V$ curvature is distributed to the interior face corners of K. Because   $V - E + F = 1$, this is $F - 1$. On the other hand, by Proposition 6, the total amount of curvature distributed to the interior face corners is at least $X \times (1/3 + 1/960n^4$) where $X$ is the number of interior face corners. Since every face is a triangle, $X$ is 3 $\times$ the number of interior faces or $3  \times (F - B)$. So we have the inequality
\begin{equation*}
F - 1 \geq 3 \times (F - B) \times (1/3 + \epsilon/3) 
\end{equation*}
where $\epsilon = 1/320n^4$.
Collecting terms together and dividing by $\epsilon$ gives
\begin{equation*}
F \leq B \times (1 + 1/\epsilon) - 1/\epsilon    
\end{equation*}
Then, since $Area = F$, and $Length \geq B$, this gives
\begin{equation*}
Area \leq Length \times (1 + 1/\epsilon) - 1/\epsilon
\end{equation*}
This shows that $F(2,n)$ satisfies a \textit{linear isoperimetric inequality} relating $Area$ and $Length$. This is one of the ‘equivalent definitions’ of hyperbolicity (see, for example, Theorems 6.5.3 and 6.6.1 in \cite{hrr}) and so we deduce that $F(2,n)$ is hyperbolic.
\end{proof}
\end{theorem}

\noindent \textbf{Acknowledgements} 
I would like to thank Martin Edjvet for his careful reading of an early draft of this paper and providing many constructive comments for improvement. I would also like to thank Caitlin for her generous help and advice.

\bibliography{references.bib}

\end{document}